\documentclass{article}
\usepackage{anysize}
\marginsize{1in}{1in}{1in}{1in}
\usepackage[utf8]{inputenc}
\usepackage{bbold}
\title{Residues of Terms of Lucas Sequences Modulo $3^k$}
\author{J.C. Saunders and R. Nicholas Stephens}
\usepackage{amssymb}
\usepackage{amsmath}
\usepackage{amsthm}
\usepackage{mathtools}
\theoremstyle{definition}

\newtheorem{thm}{Theorem}
\newtheorem{lemma}{Lemma}
\newtheorem{prop}{Proposition}
\newtheorem{notation}{Notation}

\newtheorem{note}{Note}
\theoremstyle{definition}
\DeclareMathOperator{\ord}{ord}

\begin{document}
\maketitle
\begin{abstract}
The Fibonacci sequence defined by $F_0=0$, $F_1=1$, and $F_n=F_{n-1}+F_{n-2}$ has a shortest period length of $4\cdot 3^{k-1}$ modulo $3^k$ for every $k\in\mathbb{N}$. In 2011, Bundschuh and Bundschuh \cite{bundschuh3} gave the frequencies of every residue $0\leq b\leq 3^k-1$ in this shortest period. In particular, their result implies that the Fibonacci sequences is not stable modulo $3$. Here we extend this result to other Lucas sequences. More specifically, we give analogous results for Lucas sequences defined by $\left(u_n\right)_n$ with $u_0=0$, $u_1=1$, and $u_n=Pu_{n-1}+u_{n-2}$ for all $n\geq 2$, as well as Lucas sequences defined by $\left(v_n\right)_n$ with $v_0=2$, $v_1=P$, and $v_n=Pv_{n-1}+v_{n-2}$ for all $n\geq 2$. In particular, our result implies that none of these Lucas sequences are stable modulo $3$ either.
\end{abstract}
\section{Introduction}
It is well-known that for any fixed positive integer $m$ that if the reduced residues of the terms of any linear recurrence sequence are taken modulo $m$, then the sequence eventually repeats \cite{ward}. Such a result prompted the study of the frequency of every possible reduced residue $0,1,2,\ldots,m-1$ in the period of any linear recurrence sequence modulo any integer $m$ and how these frequencies compare to each other. For instance, the Fibonacci sequence is defined by $F_0=0$, $F_1=1$, and $F_n=F_{n-1}+F_{n-2}$ for all $n\geq 2$. Jacobson \cite{jacobson} defined $v(m,b)$ to be the frequency of any residue $b$ modulo $m$ in a full period of the Fibonacci sequences modulo $m$ and specifically found the values of $v(2^k,b)$ for every $0\leq b\leq 2^k-1$ and every $k\in\mathbb{N}$.\par
This result prompted Carlip and Jacobson \cite{carlip} to generalize the definition of $v(m,b)$ to other Lucas sequences and define the concept of stability. For any sequence $\{u_i\}_i$ defined by $u_0=0$, $u_1=1$, and $u_i=au_{i-1}+bu_{i-2}$ for all $i\geq 2$, where $a$ and $b$ are fixed integers, they denoted the frequency of any residue $r$ modulo $m$ in the shortest period by $v_{a,b}(m,r)$ and defined
\begin{equation*}
\Omega_{a,b}(m):=\{v_{a,b}(m,r)|r\in\{0,\ldots,m-1\}\}.
\end{equation*}
They also defined the sequence being stable modulo a prime number $p$ if there exists a positive integer $k$ such that for every integer $N\geq k$, we have $\Omega(p^N)=\Omega(p^k)$. For example, Jacobson's result implies that the Fibonacci sequence is stable modulo $2$. Carlip and Jacobson also gave analogous results for these generalized sequences, for values of $a$ and $b$ satisfying certain congruence conditions modulo $16$, determining stability modulo $2$ also holds in the sequences considered. Morgan \cite{morgan} demonstrated, however, that the Lucas sequence, which is defined by $L_0=2$, $L_1=1$, and $L_n=L_{n-1}+L_{n-2}$ for all $n\geq 2$ is not stable modulo $2$ and determined all of the frequencies of every reduced residue modulo any given power of $2$.\par
The stability of linear recurrences modulo other primes has also been investigated. For instance, Carroll, Jacobson, and Somer \cite{carroll} proved that if a second-order linear recurrence sequence is defined by $x_0=0$, $x_1=1$, and $x_n=Ax_{n-1}+x_{n-2}$ for every $n\geq 2$, where $A$ is any fixed integer, then the set of frequencies of every reduced residue modulo $p^e$ is $\{0,2,4\}$, where $p^e$ is any prime power with $p>7$. In particular, every such sequence is stable for every prime $p>7$. Bundschuh and Bundschuh \cite{bundschuh2},\cite{bundschuh3} determined the frequencies of every reduced residue in the Lucas sequence modulo powers of $3$ and $5$, as well as the frequencies of every reduced residue in the Fibonacci sequences modulo all powers of $3$. The results indicated that the Lucas sequence was not stable modulo $5$ and neither sequence was stable modulo $3$. Similar results have been proven for a wide variety of other second-order linear recurrences as well (see, for instance, \cite{carlip3},\cite{carlip2},\cite{somer}).\par
Here we extend these results and determine the set of frequences of every residue modulo $3^k$ for every $k\in\mathbb{N}$ in Lucas sequences $(u_n)_n$ defined by $u_0=0$, $u_1=1$, and $u_n=Pu_{n-1}+u_{n-2}$ for every $n\geq 2$ and $(v_n)_n$ defined by $v_0=2$, $v_1=P$, and $v_n=Pv_{n-1}+v_{n-2}$ for every $n\geq 2$ for every fixed integer $P$. Our method of proof follows that of Bundschuh and Bundschuh \cite{bundschuh3}. Let
\begin{equation*}
v_{\mathcal{A}}(m,b):=\#\{n|0\leq n<h_{\mathcal{A}},a_n\equiv b\pmod m\},
\end{equation*}
where $\mathcal{A}$ is the linear recurrence in question. We divide into two cases: $3\mid P$ and $3\nmid P$. Throughout this paper, we use the following notation.
\begin{notation}
For a prime $p$ and a number $n$, we write $p^k\mathrel\Vert n$ and $\ord_p(n)=k$ if $p^k$ is the highest power of $p$ dividing $n$.
\end{notation}
\begin{thm}\label{thm1}
Let $3\nmid P$ and $3^{\delta}\mathrel\Vert P^2+2$.
\newline
\newline
1) Suppose $\delta=1$ and $k\in\mathbb{N}$. Then for every residue $b$ we have
\begin{equation}\label{thm1case1u}
v_{u}(3^k,b)=
\begin{cases}
3^{\lfloor k/2\rfloor}+2, & \text{if }b\equiv\pm u_{2\cdot 3^{2\lfloor(k-1)/4\rfloor+1}}\pmod{3^k}\\
2\cdot 3^l+2, &\text{if }b\equiv\pm u_{2\cdot 3^{l-1}}\pmod{3^{2l+1}}\text{ for some }l\in\{1,\ldots,\lfloor(k-1)/2\rfloor\}\\
2, &\text{otherwise}
\end{cases}
\end{equation}
and
\begin{equation}\label{thm1case1v}
v_{v}(3^k,b)=
\begin{cases}
3^{\lfloor k/2\rfloor}+2, & \text{if }b\equiv\pm 2\pmod{3^k}\\
2\cdot 3^l+2, &\text{if }b\equiv\pm v_{4\cdot 3^{l-1}}\pmod{3^{2l+1}}\text{ for some }l\in\{1,\ldots,\lfloor(k-1)/2\rfloor\}\\
2, &\text{otherwise.}
\end{cases}
\end{equation}
2) Suppose $\delta\geq 2$ and $k\in\mathbb{N}$, $k\geq 2\delta-1$. Then for every residue $b$ we have
\begin{equation}\label{thm1case2u}
v_{u}(3^k,b)=
\begin{cases}
3^{\lfloor k/2\rfloor}, & \text{if }b\equiv\pm u_{2\cdot 3^{2\lfloor(k-2\delta+1)/4\rfloor+1}}\pmod{3^k}\\
2\cdot 3^l, &\text{if }b\equiv\pm u_{2\cdot 3^{l-1}}\pmod{3^{2l+1}}\text{ for some }l\in\{\delta,\ldots,\lfloor(k-1)/2\rfloor\}\\
2, &\text{if }b\equiv 0,\pm 1\pmod{3^\delta}\\
0, &\text{otherwise}
\end{cases}
\end{equation}
and
\begin{equation}\label{thm1case2v}
v_{v}(3^k,b)=
\begin{cases}
3^{\lfloor k/2\rfloor}, & \text{if }b\equiv\pm 2\pmod{3^k}\\
2\cdot 3^l, &\text{if }b\equiv\pm v_{4\cdot 3^{l-1}}\pmod{3^{2l+1}}\text{ for some }l\in\{\delta,\ldots,\lfloor(k-1)/2\rfloor\}\\
2, &\text{if }b\equiv 0,\pm P\pmod{3^{\delta}}\\
0, &\text{otherwise.}
\end{cases}
\end{equation}
3) Suppose $\delta\geq 2$ and $k\in\mathbb{N}$, $\delta\leq k<2\delta-1$. Then for every residue $b$ we have
\begin{equation}\label{thm1case3u}
v_{u}(3^k,b)=
\begin{cases}
3^{k-\delta}, & \text{if }b\equiv\pm P\pmod{3^k}\\
2, &\text{if }b\equiv 0,\pm 1\pmod{3^\delta}\\
0, &\text{otherwise}
\end{cases}
\end{equation}
and
\begin{equation}\label{thm1case3v}
v_{v}(3^k,b)=
\begin{cases}
3^{k-\delta}, & \text{if }b\equiv\pm 2\pmod{3^k}\\
2, &\text{if }b\equiv 0,\pm P\pmod{3^{\delta}}\\
0, &\text{otherwise.}
\end{cases}
\end{equation}
\end{thm}
\begin{thm}\label{thm2}
Let $3\mid P$ and $3^{\delta}\mathrel\Vert P$.
\newline
\newline
1) Suppose $k\in\mathbb{N}$, $k\geq 2\delta-1$. Then for every residue $b$ we have
\begin{equation}\label{thm2case1u}
v_{u}(3^k,b)=
\begin{cases}
3^{\lfloor k/2\rfloor}, & \text{if }b\equiv u_{3^{2\lfloor(k-2\delta+1)/4\rfloor+1}}\pmod{3^k}\\
2\cdot 3^l, &\text{if }b\equiv u_{3^{l-\delta}}\pmod{3^{2l+1}}\text{ for some }l\in\{\delta,\ldots,\lfloor(k-1)/2\rfloor\}\\
1, &\text{if }b\equiv 0\pmod{3^\delta}\\
0, &\text{otherwise}
\end{cases}
\end{equation}
and
\begin{equation}\label{thm2case1v}
v_{v}(3^k,b)=
\begin{cases}
3^{\lfloor k/2\rfloor}, & \text{if }b\equiv 2\pmod{3^k}\\
2\cdot 3^l, &\text{if }b\equiv v_{2\cdot 3^{l-\delta}}\pmod{3^{2l+1}}\text{ for some }l\in\{\delta,\ldots,\lfloor(k-1)/2\rfloor\}\\
1, &\text{if }b\equiv 0\pmod{3^{\delta}}\\
0, &\text{otherwise.}
\end{cases}
\end{equation}
2) Suppose $\delta\geq 2$ and $k\in\mathbb{N}$, $\delta\leq k<2\delta-1$. Then for every residue $b$ we have
\begin{equation}\label{thm2case2u}
v_{u}(3^k,b)=
\begin{cases}
3^{k-\delta}, & \text{if }b\equiv 1\pmod{3^k}\\
1, &\text{if }b\equiv 0\pmod{3^{\delta}}\\
0, &\text{otherwise}
\end{cases}
\end{equation}
and
\begin{equation}\label{thm2case2v}
v_{v}(3^k,b)=
\begin{cases}
3^{k-\delta}, & \text{if }b\equiv 2\pmod{3^k}\\
1, &\text{if }b\equiv 0\pmod{3^{\delta}}\\
0, &\text{otherwise.}
\end{cases}
\end{equation}
\end{thm}
\section{Classical Lemmas on Lucas Sequences}
We first state some basic properties and classical lemmas on Lucas sequences. Let $D:=P^2+4$, $\alpha:=\frac{P+\sqrt{D}}{2}$ and $\beta:=\frac{P-\sqrt{D}}{2}$. We have the Binet formulas for $\left(u_n\right)_n$ $\left(v_n\right)_n$, which are
\begin{equation*}
u_n=\frac{\alpha^n-\beta^n}{\alpha-\beta}
\end{equation*}
and
\begin{equation*}
v_n=\alpha^n+\beta^n,
\end{equation*}
valid for all $n\in\mathbb{N}$. The results in the following lemmas are easy to prove and found in \cite{ballot} and \cite{lucas}.
\begin{lemma}\label{lem1}
For all $n\in\mathbb{Z}$ we have the following.
\begin{align}
u_{2n}&=u_nv_n\label{lem1.1}\\
v_{2n}&=v_n^2-2(-1)^n\label{lem1.2}\\
v_{2n}&=u_nv_{n+1}+u_{n-1}v_n\label{lem1.3}\\
u_{3n}&=u_n\left(Du_n^2+3(-1)^n\right)\label{lem1.4}.
\end{align}
Also, for all $n,q,r\in\mathbb{Z}$ with $n$ and $q$ not both zero, we have 
\begin{equation}
\gcd\left(u_{qn+r},u_n\right)=\left(u_n,u_r\right).\label{lem1.6}
\end{equation}
\end{lemma}
\begin{lemma}\label{lem2}
For all $s,t\in\mathbb{Z}$ of the same parity we have the following.
\begin{align}
u_s+u_t(-1)^{\frac{s-t}{2}}&=u_{\frac{s+t}{2}}v_{\frac{s-t}{2}}\label{lem2.1}\\
u_s-u_t(-1)^{\frac{s-t}{2}}&=u_{\frac{s-t}{2}}v_{\frac{s+t}{2}}\label{lem2.2}\\
v_s+v_t(-1)^{\frac{s-t}{2}}&=v_{\frac{s+t}{2}}v_{\frac{s-t}{2}}\label{lem2.3}\\
v_s-v_t(-1)^{\frac{s-t}{2}}&=Du_{\frac{s+t}{2}}u_{\frac{s-t}{2}}\label{lem2.4}.
\end{align}
\end{lemma}
\begin{proof}
Found in \cite{lucas}. Follow the proofs of Lemma $2.2$ in \cite{bundschuh3} and Lemma $1$ in \cite{bundschuh2}, noting that $\alpha\beta=-1$.
\end{proof}

\section{The Case of $3\nmid P$}

We deal with the case of $3\nmid P$ first.

\begin{note}
For this section let $\delta:=\ord_3\left(P^2+2\right)$.
\end{note}

\subsection{Preliminary Lemmas}

Here we determine several useful preliminary lemmas on relevant divisibility properties of the Lucas sequences, as well as their period lengths modulo $3^k$.

\begin{lemma}\label{lem3}
We have
\begin{align}
4\mid n&\Leftrightarrow 3\mid u_n\label{lem3.1}\\
4\mid n&\Rightarrow\ord_3 u_n=\ord_3 n+\ord_3\delta\label{lem3.2}\\
2\mathrel\Vert n&\Leftrightarrow 3\mid v_n\label{lem3.3}\\
2\mathrel\Vert n&\Leftrightarrow\ord_3v_n=\ord_3n+\ord_3\delta\label{lem3.4}.
\end{align}
\end{lemma}
\begin{proof}
From $u_0=0$, $u_1=1$, $u_2=P$, $u_3=P^2+1\equiv 2\pmod 3$, $u_4=P^3+2P\equiv 0\pmod 3$, and Lemma \ref{lem1} we have \eqref{lem3.1}. From $v_0=2$, $v_1=P$, $v_2=P^2+2\equiv 0\pmod 3$, $v_3\equiv P\pmod 3$, $v_4\equiv P^2\equiv 1\pmod 3$, $v_5\equiv 2P\pmod 3$, $v_6\equiv 2P^2+1\equiv 0\pmod 3$, $v_7\equiv 2P\pmod 3$, $v_8\equiv 2P^2\equiv 2\pmod 3$, $v_9\equiv 2P+2P\equiv P\pmod 3$ we can see that \eqref{lem3.3} holds.
\newline
\newline
For \eqref{lem3.2} and \eqref{lem3.4} we prove by induction on $\ord_3(n)\in\mathbb{N}_0$. First suppose $4\mid n$, but $\ord_3n=0$. Then $n\equiv 4,8\pmod{12}$. Since $3\mid u_4$, we have by \eqref{lem1.4} that $9\mid u_{12}$. By \eqref{lem1.6} we have $\gcd(u_n,u_{12})=\gcd(u_8,u_{12})=\gcd(u_4,u_{12})$. Noting that $u_4=P^3+2P=P(P^2+2)$, we have $\ord_3\left(P^2+2\right)=\ord_3(u_4)=\ord_3(u_n)$, verifying $\eqref{lem3.2}$ for $\ord_3(n)=0$. Now for all $n\in\mathbb{N}$ such that $3\mid u_n$ we can see by \eqref{lem1.4} that $\ord_3(u_{3n})=\ord_3(u_n)+1$ and \eqref{lem1.4} now follows by induction on $\ord_3(n)$.
\newline
\newline
Now suppose that $2\mathrel\Vert n$. By \eqref{lem1.1} we have $v_n=\frac{u_{2n}}{u_n}$. We have $4\mid 2n$ so by \eqref{lem3.2} we have $\ord_3u_{2n}=\ord_32n+\ord_3\left(P^2+2\right)=\ord_3n+\ord_3\left(P^2+2\right)$ and $3\nmid u_n$. Thus, \eqref{lem3.4} follows.
\end{proof}
\begin{lemma}
For all $k\geq\delta$ we have
\begin{equation*}
h_u\left(3^k\right)=h_v\left(3^k\right)=8\cdot 3^{k-\delta}.
\end{equation*}
\end{lemma}
\begin{proof}
Let $k\geq\delta$. From \eqref{lem3.1} we have $4\mid h_u\left(3^k\right)$. For all $m\in\mathbb{Z}$ we have $3^k\mid u_{4m}$ if and only if $3^{k-\delta}\mid m$ by \eqref{lem3.2}. Also, for all $m\in\mathbb{Z}$ we have $u_{4m+1}-u_1=u_{2m}v_{2m+1}$ by \eqref{lem2.2}. From Lemma \ref{lem3} we have $3^k\mid u_{2m}v_{2m+1}$ if and only if $3^k\mid u_{2m}$ if and only if $2\cdot 3^{k-\delta}\mid m$. It follows that $h_u\left(3^k\right)=8\cdot 3^{k-\delta}$.
\newline
\newline
From \eqref{lem3.3} we have $4\mid h_v\left(3^k\right)$. For all $m\in\mathbb{Z}$ we have $v_{4m}-v_0=Du_{2m}^2$ by \eqref{lem2.4}. From Lemma \ref{lem3} we have $3^k\mid Du_{2m}^2$ if and only if $3^{\lceil k/2\rceil}\mid u_{2m}$ if and only if $3^{\lceil k/2\rceil-\delta}\mid m$. Also, for all $m\in\mathbb{Z}$ we have $v_{4m+1}-v_1=Du_{2m+1}u_{2m}$ by \eqref{lem2.4}. As before, $3^k\mid u_{2m+1}u_{2m}$ if and only if $2\cdot 3^{k-\delta}\mid m$. It follows that $h_v\left(3^k\right)=8\cdot 3^{k-\delta}$.
\end{proof}
\begin{lemma}
For all $k\geq\delta$ and $n\in\mathbb{Z}$ we have 
\begin{equation}\label{lem5.1}
u_{n+4\cdot 3^{k-\delta}}\equiv-u_n\pmod{3^k}
\end{equation}
and
\begin{equation}\label{lem5.2}
v_{n+4\cdot 3^{k-\delta}}\equiv-v_n\pmod{3^k},
\end{equation}
\end{lemma}
\begin{proof}
Let $k\geq\delta$ and $n\in\mathbb{Z}$. By \eqref{lem2.1} we have
\begin{equation*}
u_{n+4\cdot 3^{k-\delta}}+u_n=u_{n+2\cdot 3^{k-\delta}}v_{2\cdot 3^{k-\delta}}.
\end{equation*}
Since $2\mathrel\Vert 2\cdot 3^k$ we have $\ord_3\left(v_{2\cdot 3^{k-\delta}}\right)=k-\delta+\delta=k$ by \eqref{lem3.4}. Thus, \eqref{lem5.1} follows.
\newline
\newline
By \eqref{lem2.3} we have
\begin{equation*}
v_{n+4\cdot 3^{k-\delta}}+v_n=v_{n+2\cdot 3^{k-\delta}}v_{2\cdot 3^{k-\delta}}.
\end{equation*}
Thus, \eqref{lem5.2} follows again by \eqref{lem3.4}.
\end{proof}
\begin{lemma}
For all $k\geq\delta$ and $n\in\mathbb{Z}$ we have
\begin{equation}\label{lem6}
u_{3^{k-\delta}}v_n\equiv v_{3^{k-\delta}}u_{n-2\cdot 3^{k-\delta}}\pmod{3^k}.
\end{equation}
\end{lemma}
\begin{proof}
Fix $k\geq\delta$. Since $u_{3^{k-\delta}}=u_{-3^{k-\delta}}$ we have
\begin{equation*}
u_{3^{k-\delta}}v_{3^{k-\delta}}=v_{3^{k-\delta}}u_{-3^{k-\delta}}
\end{equation*}
and \eqref{lem6} holds for $n=3^{k-\delta}$. From \eqref{lem1.3} we have
\begin{equation*}
u_{3^{k-\delta}}v_{3^{k-\delta}+1}+u_{3^{k-\delta}-1}v_{3^{k-\delta}}=v_{2\cdot 3^{k-\delta}}.
\end{equation*}
Since $u_{3^{k-\delta}-1}=-u_{1-3^{k-\delta}}$, it follows that
\begin{equation*}
u_{3^{k-\delta}}v_{3^{k-\delta}+1}-u_{1-3^{k-\delta}}v_{3^{k-\delta}}=v_{2\cdot 3^{k-\delta}}.
\end{equation*}
Also, \eqref{lem3.4} gives $\ord_3v_{2\cdot 3^{k-\delta}}=k$. Thus, we have \eqref{lem6} for $n=3^{k-\delta}+1$. From the recurrence relations we have \eqref{lem6} for all $n\in\mathbb{Z}$.
\end{proof}
\subsection{The Proof}
We now give more significant lemmas and then derive various propositions explicitly comparing the residues of terms of our sequences modulo the relevant powers of $3$, leading to the proof of Theorem \ref{thm1}.
\begin{note}
For ease of notation, we let $x=2\delta-1$ and $J(k):=2\cdot 3^{2\lfloor(k-x)/4\rfloor+1}$.
\end{note}
\begin{lemma}\label{lem7}
Let $k\in\mathbb{N}$, $k\geq 2\delta-1$, and $n$ be even. Then
\begin{equation}\label{lem7.1}
u_n\equiv u_{J(k)}\pmod{3^k}
\end{equation}
if and only if $n=J(k)+8\cdot 3^{\left\lfloor(k-x)/2\right\rfloor}j$ for some $j\in\mathbb{Z}$. Also,
\begin{equation}\label{lem7.2}
v_n\equiv v_0\pmod{3^k}.
\end{equation}
if and only if $n=8\cdot 3^{\left\lfloor(k-x)/2\right\rfloor}j$ for some $j\in\mathbb{Z}$.
\end{lemma}
\begin{proof}
By \eqref{lem2.2} we have
\begin{equation*}
u_{J(k)+8\cdot 3^{\left\lfloor(k-x)/2\right\rfloor}j}-u_{J(k)}=u_{4\cdot 3^{\left\lfloor(k-x)/2\right\rfloor}j}v_{J(k)+4\cdot 3^{\left\lfloor(k-x)/2\right\rfloor}j}.
\end{equation*}
By \eqref{lem3.2} we have $\ord_3\left(u_{4\cdot 3^{\left\lfloor(k-x)/2\right\rfloor}j}\right)=\left\lfloor(k-x)/2\right\rfloor+\delta$. Also, by \eqref{lem3.4} we have
$\ord_3\left(v_{J(k)+4\cdot 3^{\left\lfloor(k-x)/2\right\rfloor}j}\right)\geq\left\lfloor(k-x)/2\right\rfloor+\delta$. Thus, $\ord_3\left(u_{J(k)+8\cdot 3^{\left\lfloor(k-x)/2\right\rfloor}j}-u_{J(k)}\right)=2\left\lfloor(k-x)/2\right\rfloor+2\delta\geq k-x-1+2\delta=k$. Thus, \eqref{lem7.1} follows.
\newline
\newline
For the converse, suppose $\eqref{lem7.1}$ holds for some even $n$. Let $n=2m$. Since $3\nmid u_{J(k)}$, we have $3\nmid u_n$, so that $m$ is odd by \eqref{lem3.1}. By \eqref{lem2.2} we have $u_n-u_{J(k)}=u_{m-3^{2\lfloor(k-x)/4\rfloor+1}}v_{m+3^{2\lfloor(k-x)/4\rfloor+1}}$. Since $3^k\mid u_n-u_{J(k)}$ we have either $3^{\lfloor(k+1)/2\rfloor}\mid u_{m-3^{2\lfloor(k-x)/4\rfloor+1}}$ or $3^{\lfloor(k+1)/2\rfloor}\mid v_{m+3^{2\lfloor(k-x)/4\rfloor+1}}$. Suppose $3^{\lfloor(k+1)/2\rfloor}\mid u_{m-3^{2\lfloor(k-x)/4\rfloor+1}}$. Then we have $m-3^{2\lfloor(k-x)/4\rfloor+1}=4\cdot 3^{\lfloor(k+1)/2\rfloor-\delta}j$ for some $j\in\mathbb{Z}$ by \eqref{lem3.1} and \eqref{lem3.2}. Thus, $n=J(k)+8\cdot 3^{\left\lfloor(k-x)/2\right\rfloor}j$. Suppose $3^{\lfloor(k+1)/2\rfloor}\mid v_{m+3^{2\lfloor(k-x)/4\rfloor+1}}$. Then we have $m+3^{2\lfloor(k-x)/4\rfloor+1}=(4j+2)\cdot 3^{\lfloor(k+1)/2\rfloor-\delta}$ for some $j\in\mathbb{Z}$ by \eqref{lem3.3} and \eqref{lem3.4}. Thus,
\begin{align*}
n&=-J(k)+8\cdot 3^{\left\lfloor(k-x)/2\right\rfloor}j+4\cdot 3^{\left\lfloor(k-x)/2\right\rfloor}\\
&=J(k)+8\cdot 3^{\left\lfloor(k-x)/2\right\rfloor}j+4\cdot 3^{\left\lfloor(k-x)/2\right\rfloor}-4\cdot 3^{2\lfloor(k-x)/4\rfloor+1}.
\end{align*}
Either $2\lfloor(k-x)/4\rfloor+1-\left\lfloor(k-x)/2\right\rfloor=0$ or $1$. If it is $0$, then we have $n=J(k)+8\cdot 3^{\left\lfloor(k-x)/2\right\rfloor}j$. If it is $1$, then we have $n=J(k)+8\cdot 3^{\left\lfloor(k-x)/2\right\rfloor}(j-1)$.
\newline
\newline
By \eqref{lem2.4} we have 
\begin{equation*}
v_{8\cdot 3^{\left\lfloor(k-x)/2\right\rfloor}j}-v_0=Du_{4\cdot 3^{\left\lfloor(k-x)/2\right\rfloor}j}^2.
\end{equation*}
Using $\ord_3\left(u_{4\cdot 3^{\left\lfloor(k-x)/2\right\rfloor}j}\right)\geq\left\lfloor(k-x)/2\right\rfloor+\delta$ once again we obtain \eqref{lem7.2}.
\newline
\newline
For the converse, suppose \eqref{lem7.2} holds for some even $n$. Since $3\nmid v_0$, we have that $3\nmid v_n$. Thus, $n$ is a multiple of $4$ by \eqref{lem3.3}. Let $n=4m$. By \eqref{lem2.4} we have 
\begin{equation*}
v_n-v_0=Du_{2m}^2
\end{equation*}
Thus, $3^{\lfloor(k+1)/2\rfloor}\mid u_{2m}$. It follows that $m=2\cdot 3^{\left\lfloor(k-x)/2\right\rfloor}j$ for some $j\in\mathbb{Z}$ by \eqref{lem3.1}. Thus, $n=8\cdot 3^{\left\lfloor(k-x)/2\right\rfloor}j$.
\end{proof}
\begin{prop}\label{prop1}
Let $k\in\mathbb{N}$, $k\geq 2\delta-1$. Each of the four congruences
\begin{align*}
u_n\equiv&\pm u_{J(k)}\pmod{3^k}\\
v_n\equiv&\pm v_0\pmod{3^k}
\end{align*}
has exactly $3^{\lfloor k/2\rfloor}$ even solutions $n\in\left\{0,\ldots,8\cdot 3^{k-\delta}-1\right\}$. Moreover, for the first two congruences all of these solutions satisfy $n\equiv 2\pmod 4$ and for the second two congruences all of these solutions satisfy $n\equiv 0\pmod 4$.
\end{prop}
\begin{proof}
First, by Lemma \ref{lem7}, we have that the set of even solutions to $u_n\equiv u_{J(k)}\pmod{3^k}$ are exactly all integers of the form $J(k)+8\cdot 3^{\left\lfloor(k-x)/2\right\rfloor}j$, $j\in\mathbb{Z}$. Notice that $0\leq J(k)+8\cdot 3^{\left\lfloor(k-x)/2\right\rfloor}j<8\cdot 3^{k-\delta}$ if and only if
\begin{equation*}
-\frac{2\cdot 3^{2\lfloor(k-x)/4\rfloor+1}}{8\cdot 3^{\left\lfloor(k-x)/2\right\rfloor}}\leq j<\frac{8\cdot 3^{k-\delta}-J(k)}{8\cdot 3^{\left\lfloor(k-x)/2\right\rfloor}}=3^{\lfloor k/2\rfloor}-\frac{2\cdot 3^{2\lfloor(k-x)/4\rfloor+1}}{8\cdot 3^{\left\lfloor(k-x)/2\right\rfloor}}
\end{equation*}
where we used 
\begin{equation*}
k-\delta-\left\lfloor(k-x)/2\right\rfloor=k-\left\lfloor(k+1)/2\right\rfloor=\left\lfloor k/2\right\rfloor.
\end{equation*}
Counting all such values of $j$ gives the result for the first congruence.
\newline
\newline
For the second congruence $u_n\equiv -u_{J(k)}\pmod{3^k}$ first note that by \eqref{lem5.1} and \eqref{lem7.1} all integers of the form $J(k)+4\cdot 3^{k-\delta}+8\cdot 3^{\left\lfloor(k-x)/2\right\rfloor}j$ are exactly all of the solutions. Arguing as in the first congruence, the appropriate values of $j$ are those satisfying
\begin{equation*}
-\frac{3^{\lfloor k/2\rfloor}}{2}-\frac{2\cdot 3^{2\lfloor(k-x)/4\rfloor+1}}{8\cdot 3^{\left\lfloor(k-x)/2\right\rfloor}}=\frac{-4\cdot 3^{k-\delta}-J(k)}{8\cdot 3^{\left\lfloor(k-x)/2\right\rfloor}}\leq j<\frac{4\cdot 3^{k-\delta}-J(k)}{8\cdot 3^{\left\lfloor(k-x)/2\right\rfloor}}=\frac{3^{\lfloor k/2\rfloor}}{2}-\frac{2\cdot 3^{2\lfloor(k-x)/4\rfloor+1}}{8\cdot 3^{\left\lfloor(k-x)/2\right\rfloor}}
\end{equation*}
Again, we get exactly $3^{\lfloor k/2\rfloor}$ appropriate values of $j$ and the result follows for the second congruence.
\newline
\newline
The proofs for the third and fourth congruences are very similar, using \eqref{lem5.2} and \eqref{lem7.2}.
\end{proof}
\begin{lemma}\label{lem8}
Let $l\in\mathbb{N}$, $l\geq\delta$, and $n$ be even. Then
\begin{equation}\label{lem8.1}
u_n\equiv u_{2\cdot 3^{l-\delta}}\pmod{3^{2l+1}}.
\end{equation}
if and only if $n\equiv 2\cdot 3^{l-\delta}\pmod{8\cdot 3^{l-\delta+1}}$ or $n\equiv 10\cdot 3^{l-\delta}\pmod{8\cdot 3^{l-\delta+1}}$.
\newline
\newline
Also, 
\begin{equation}\label{lem8.2}
v_n\equiv v_{4\cdot 3^{l-\delta}}\pmod{3^{2l+1}}.
\end{equation}
if and only if $n\equiv 4\cdot 3^{l-\delta}\pmod{8\cdot 3^{l-\delta+1}}$ or $n\equiv -4\cdot 3^{l-\delta}\pmod{8\cdot 3^{l-\delta+1}}$.
\end{lemma}
\begin{proof}
If $i=1,5$, then we have
\begin{equation*}
\ord_3\left(u_{2\cdot 3^{l-\delta}i+8\cdot 3^{l-\delta+1}j}-u_{2\cdot 3^{l-\delta}}\right)=\ord_3\left(u_{3^{l-\delta}(i-1)+4\cdot 3^{l-\delta+1}j}v_{3^{l-\delta}(i+1)+4\cdot 3^{l-\delta+1}j}\right)\geq 2l+1
\end{equation*}
by \eqref{lem2.2}, \eqref{lem3.2}, and \eqref{lem3.4}. More specifically, if $i=1$, then $\ord_3\left(u_{3^{l-\delta}(i-1)+4\cdot 3^{l-\delta+1}j}\right)\geq l+1$ and 
\newline
$\ord_3\left(v_{3^{l-\delta}(i+1)+4\cdot 3^{l-\delta+1}j}\right)=l$, while if $i=5$, then $\ord_3\left(u_{3^{l-\delta}(i-1)+4\cdot 3^{l-\delta+1}j}\right)=l$ and
\newline
$\ord_3\left(v_{3^{l-\delta}(i+1)+4\cdot 3^{l-\delta+1}j}\right)\geq l+1$. Hence, we have \eqref{lem8.1}.
\newline
\newline
For the converse, suppose $\eqref{lem8.1}$ holds for some even $n$. Let $n=2m$. Since $3\nmid u_{2\cdot 3^{l-\delta}}$, we have $3\nmid u_n$, so that $m$ is odd by \eqref{lem3.1}. By \eqref{lem2.2} we have $u_n-u_{2\cdot 3^{l-\delta}}=u_{m-3^{l-\delta}}v_{m+3^{l-\delta}}$. Since $3^{2l+1}\mid u_n-u_{2\cdot 3^{l-\delta}}$ we have either $3^{l+1}\mid u_{m-3^{l-\delta}}$ or $3^{l+1}\mid v_{m+3^{l-\delta}}$. Suppose $3^{l+1}\mid u_{m-3^{l-\delta}}$. Then we have $m-3^{l-\delta}=4\cdot 3^{l-\delta+1}j$ for some $j\in\mathbb{Z}$ by \eqref{lem3.1} and \eqref{lem3.2}. Thus, $n=2\cdot 3^{l-\delta}+8\cdot 3^{l-\delta+1}j$. Suppose $3^{l+1}\mid v_{m+3^{l-\delta}}$. Then we have $m+3^{l-\delta}=(4j+2)\cdot 3^{l-\delta+1}$ for some $j\in\mathbb{Z}$ by \eqref{lem3.3} and \eqref{lem3.4}. Thus, $n=10\cdot 3^{l-\delta}+8\cdot 3^{l-\delta+1}j$.
\newline
\newline
Similarly, if $i=\pm 1$, we have
\begin{equation*}
\ord_3\left(v_{4\cdot 3^{l-\delta}i+8\cdot 3^{l-\delta+1}j}-v_{4\cdot 3^{l-\delta}}\right)=\ord_3\left(Du_{2\cdot 3^{l-\delta}(i-1)+4\cdot 3^{l-\delta+1}j}u_{2\cdot 3^{l-\delta}(i+1)+4\cdot 3^{l-\delta+1}j}\right)\geq 2l+1
\end{equation*}
by \eqref{lem2.4}, \eqref{lem3.3}, and \eqref{lem3.4} with both cases considered separately as before. Hence, we have \eqref{lem8.2}.
\newline
\newline
For the converse, suppose $\eqref{lem8.2}$ holds for some even $n$. First, since $3\mid v_n-v_{4\cdot 3^{l-\delta}}$, but $3\nmid v_{4\cdot 3^{l-\delta}}$ by \eqref{lem3.3} we have $3\nmid v_n$. It follows that $4\mid n$ by \eqref{lem3.3}. Let $n=4m$. By \eqref{lem2.4} we have $v_n-v_{4\cdot 3^{l-\delta}}=Du_{2m-2\cdot 3^{l-\delta}}u_{2m+2\cdot 3^{l-\delta}}$. Since $3^{2l+1}\mid v_n-v_{4\cdot 3^{l-\delta}}$ we have either $3^{l+1}\mid u_{2m-2\cdot 3^{l-\delta}}$ or $3^{l+1}\mid u_{2m+2\cdot 3^{l-\delta}}$. Suppose $3^{l+1}\mid u_{2m-2\cdot 3^{l-\delta}}$. Then we have $2m-2\cdot 3^{l-\delta}=4\cdot 3^{l-\delta+1}j$ for some $j\in\mathbb{Z}$ by \eqref{lem3.1} and \eqref{lem3.2}. Thus, $n=4\cdot 3^{l-\delta}+8\cdot 3^{l-\delta+1}j$. Suppose $3^{l+1}\mid u_{2m+2\cdot 3^{l-\delta}}$. Then we have $2m+2\cdot 3^{l-\delta}=4\cdot 3^{l-\delta+1}j$ for some $j\in\mathbb{Z}$ by \eqref{lem3.3} and \eqref{lem3.4}. Thus, $n=-4\cdot 3^{l-\delta}+8\cdot 3^{l-\delta+1}j$.
\end{proof}
\begin{prop}\label{prop2}
Let $k\geq 2\delta-1$.
\newline
1) Let $n\in\mathbb{N}$ be even such that $u_n\equiv u_{2\cdot 3^{l-\delta}}\pmod{3^{2l+1}}$, where $\delta\leq l\leq\lfloor\frac{k-1}{2}\rfloor$. Then there exists exactly $2\cdot 3^l$ even $m\in\{0,\ldots,8\cdot 3^{k-\delta}-1\}$ such that $u_n\equiv u_m\pmod{3^k}$. The same assertion holds for the congruence $u_n\equiv -u_{2\cdot 3^{l-\delta}}\pmod{3^{2l+1}}$. Moreover all of these solutions satisfy $n\equiv 2\pmod 4$.
\newline
\newline
2) Let $n\in\mathbb{N}$ be even such that $v_n\equiv v_{4\cdot 3^{l-\delta}}\pmod{3^{2l+1}}$, where $\delta\leq l\leq\lfloor\frac{k-1}{2}\rfloor$. Then there exists exactly $2\cdot 3^l$ even $m\in\{0,\ldots,8\cdot 3^{k-\delta}-1\}$ such that $v_n\equiv v_m\pmod{3^k}$. The same assertion holds for the congruence $v_n\equiv -v_{4\cdot 3^{l-\delta}}\pmod{3^{2l+1}}$. Moreover all of these solutions satisfy $n\equiv 0\pmod 4$.
\end{prop}
\begin{proof}
1) Let $n\in\mathbb{N}$ be even such that $u_n\equiv u_{2\cdot 3^{l-\delta}}\pmod{3^{2l+1}}$. By Lemma \ref{lem8} $n$ has the form $a\cdot 3^{l-\delta}+8\cdot 3^{l-\delta+1}j$, where $a=2,10$. If $u_n\equiv u_m\pmod{3^k}$, then $u_n\equiv u_m\pmod{3^{2l+1}}$, so that $m$ must be of the form $b\cdot 3^{l-\delta}+8\cdot 3^{l-\delta+1}r$, where $b=2,10$, again by Lemma \ref{lem8}. For all such values of $m$ we have
\begin{equation*}
u_n-u_m=u_{\frac{(a-b)\cdot 3^{l-\delta}}{2}+4\cdot 3^{l-\delta+1}(j-r)}v_{\frac{(a+b)\cdot 3^{l-\delta}}{2}+4\cdot 3^{l-\delta+1}(j+r)}.
\end{equation*}
Suppose $a\neq b$. Then by \eqref{lem3.2} $3^l\mathrel\Vert u_{\frac{(a-b)\cdot 3^{l-\delta}}{2}+4\cdot 3^{l-\delta+1}(j-r)}$. Thus $3^k\mid u_n-u_m$ if and only if $3^{k-l}\mid v_{\frac{(a+b)\cdot 3^{l-\delta}}{2}+4\cdot 3^{l-\delta+1}(j+r)}=v_{2\cdot 3^{l-\delta+1}(2j+2r+1)}$. By \eqref{lem3.4} we have 
\newline
$\ord_3\left(v_{2\cdot 3^{l-\delta+1}(2j+2r+1)}\right)=\ord_3\left(2\cdot 3^{l-\delta+1}(2j+2r+1)\right)+\delta=l+1+\ord_3(2j+2r+1)$. Thus, $3^k\mid u_n-u_m$ if and only if $3^{k-2l-1}\mid 2j+2r+1$. Since $0\leq r\leq 3^{k-l-1}-1$ there are $\frac{3^{k-l-1}}{3^{k-2l-1}}=3^l$ possible values of $m$.
\newline
\newline
Suppose $a=b$. Then $3^l\mathrel\Vert v_{\frac{(a+b)\cdot 3^{l-\delta}}{2}+4\cdot 3^{l-\delta+1}(j+r)}$ and similarly to the above argument we can derive that $3^k\mid u_n-u_m$ if and only if $3^{k-2l-1}\mid j-r$, again leading to $3^l$ possible values of $m$. Thus, there are $2\cdot 3^l$ possible values of $m$ in total. The assertion for the congruence $u_n\equiv -u_{2\cdot 3^{l-\delta}}\pmod{3^{2l+1}}$ can be argued similarly, using \eqref{lem5.1}.
\newline
\newline
Statement 2) follows similarly, using \eqref{lem2.4}, \eqref{lem3.3}, \eqref{lem3.4}, and Lemma \ref{lem8}.
\end{proof}
\begin{lemma}\label{lem9}
Let $k\in\mathbb{N}$ and $0\leq n\leq 8\cdot 3^{k-\delta}-1$ with $n\not\equiv 2\pmod 4$ ($\not\equiv 0\pmod 4$, respectively) and suppose $u_n\equiv b\pmod{3^k}$ ($v_n\equiv b\pmod{3^k}$, respectively), where $0\leq b\leq 3^k-1$. Then the $u_{n+8\cdot 3^{k-\delta}j}$ ($v_{n+8\cdot 3^{k-\delta}j}$, respectively), $j=0,1,2$ are congruent to $b+3^k\lambda$, $\lambda=0,1,2$, modulo $3^{k+1}$ in some order.
\end{lemma}
\begin{proof}
By \eqref{lem2.2} and \eqref{lem2.4} we have
\begin{equation*}
u_{n+8\cdot 3^{k-\delta}j}-u_{n+8\cdot 3^{k-\delta}i}=u_{4\cdot 3^{k-\delta}(j-i)}v_{n+4\cdot 3^{k-\delta}(j+i)}
\end{equation*}
and
\begin{equation*}
v_{n+8\cdot 3^{k-\delta}j}-v_{n+8\cdot 3^{k-\delta}i}=Du_{4\cdot 3^{k-\delta}(j-i)}u_{n+4\cdot 3^{k-\delta}(j+i)}
\end{equation*}
for all pairs of integers $0\leq i<j\leq 2$. First suppose that $n\not\equiv 2\pmod 4$. In the first equation we have $3\nmid v_{n+4\cdot 3^{k-\delta}(j+i)}$ by \eqref{lem3.3}. Also, $3^{k-\delta}\mathrel\Vert 4\cdot 3^{k-\delta}(j+i)$, so that $3^{k}\mathrel\Vert u_{4\cdot 3^{k-\delta}(j-i)}$. Therefore, $3^k\mathrel\Vert u_{n+8\cdot 3^{k-\delta}j}-u_{n+8\cdot 3^{k-\delta}i}$, so the result on $n\not\equiv 2\pmod 4$ follows. The case of $n\not\equiv 0\pmod 4$ is the same.
\end{proof}

For the next proposition, we will need the following notation.

\begin{notation}
Let $v_u^{\not\equiv 2}\left(3^k,b\right)$ and $v_u^{\equiv 2}\left(3^k,b\right)$ denote the number of indices counted in $v_u\left(3^k,b\right)$ that are $\not\equiv 2\pmod 4$ and $\equiv 2\pmod 4$, respectively. Also, let $v_v^{\not\equiv 0}\left(3^k,b\right)$ and $v_v^{\equiv 0}\left(3^k,b\right)$ denote the number of indices counted in $v_v\left(3^k,b\right)$ that are $\not\equiv 0\pmod 4$ and $\equiv 0\pmod 4$, respectively. 
\end{notation}

\begin{prop}\label{prop3}
For all $k\geq\delta$, if $b\equiv 0,\pm 1\pmod{3^{\delta}}$, then $v_u^{\not\equiv 2}\left(3^k,b\right)=2$, and if $b\equiv 0,\pm P\pmod{3^{\delta}}$, then $v_v^{\not\equiv 0}\left(3^k,b\right)=2$. Also, if $b\not\equiv 0,\pm 1\pmod{3^{\delta}}$, then $v_u^{\not\equiv 2}\left(3^k,b\right)=0$, and if $b\not\equiv 0,\pm P\pmod{3^{\delta}}$, then $v_v^{\not\equiv 0}\left(3^k,b\right)=0$. As well, if $b\not\equiv 0,\pm 1,\pm P\pmod{3^{\delta}}$, then $v_u\left(3^k,b\right)=0$, and if $b\not\equiv 0,\pm 2,\pm P\pmod{3^{\delta}}$, then $v_v\left(3^k,b\right)=0$.
\newline
\newline
Suppose $\delta\geq 2$. For every residue $b$ we either have $v_u\left(3^k,b\right)=v_u^{\not\equiv 2}\left(3^k,b\right)$ or $v_u\left(3^k,b\right)=v_u^{\equiv 2}\left(3^k,b\right)$. Also, for every residue $b$ we either have $v_v\left(3^k,b\right)=v_v^{\not\equiv 0}\left(3^k,b\right)$ or $v_v\left(3^k,b\right)=v_v^{\equiv 0}\left(3^k,b\right)$.
\end{prop}

\begin{proof}
First, using the fact that $3^\delta\mid P^2+2$ we obtain the following periods of $(u_n)_n$ and $(v_n)_n$ modulo $3^k$:
\begin{equation*}
u_0\equiv 0, u_1\equiv 1, u_2\equiv P, u_3\equiv -1, u_4\equiv 0, u_5\equiv -1, u_6\equiv -P, u_7\equiv 1, u_8\equiv 0, u_9\equiv 1.
\end{equation*}
\begin{equation*}
v_0\equiv 2, v_1\equiv P, v_2\equiv 0, v_3\equiv P, v_4\equiv -2, v_5\equiv -P, v_6\equiv 0, v_7\equiv -P, v_8\equiv 2, v_9\equiv P.
\end{equation*}
We can therefore see that $v_u^{\not\equiv 2}\left(3^{\delta},1\right)=v_u^{\not\equiv 2}\left(3^{\delta},-1\right)=v_u^{\not\equiv 2}\left(3^{\delta},0\right)=2$ and $v_v^{\neq}\left(3^{\delta},P\right)=v_v^{\neq}\left(3^{\delta},-P\right)=v_v^{\neq}\left(3^{\delta},0\right)=2$. Also, if $b\not\equiv 0,\pm 1\pmod{3^{\delta}}$, then $v_u^{\not\equiv 2}\left(3^{\delta},b\right)=0$, and if $b\not\equiv 0,\pm P\pmod{3^{\delta}}$, then $v_v^{\not\equiv 0}\left(3^{\delta},b\right)=0$. The first part now follows through an induction argument, using Lemma \ref{lem9}.
\newline
\newline
Suppose $\delta\geq 2$. Then since $3^{\delta}\mid P^2+2$ we have $3^\delta\nmid P^2-1=(P-1)(P+1)$. Thus, $0,\pm 1,\pm P$ are distinct residues mod $3^{\delta}$ and the statement for the sequence $u$ follows for $k=\delta$. So $n\equiv 2\pmod 4$ and $m\not\equiv 2\pmod 4$ implies that $u_m\not\equiv u_n\pmod{3^{\delta}}$ and so $u_m\not\equiv u_n\pmod{3^k}$ for all $k\geq\delta$. The statement for the sequence $u$ now follows for every $k\geq\delta$. The statement for the sequence $v$ follows similarly.
\end{proof}

We now prove Theorem \ref{thm1}.

\begin{proof}
We divide into each case of Theorem \ref{thm1}.
\newline
\newline
1) Let $k\in\mathbb{N}$. Since $\delta=1$, we have $v_u^{\not\equiv 2}\left(3^k,b\right)=2$ for every $0\leq b\leq 3^k-1$ by Proposition \ref{prop3}. Combining this observation with Propositions \ref{prop1} and \ref{prop2} gives us the first two lines of \eqref{thm1case1u}. It remains to show the third line. Let $S$ be the set of residues that are accounted for by the first two lines of \eqref{thm1case1u}. By Propositions \ref{prop1} and \ref{prop2} we have
\begin{align*}
&\quad\sum_{b\in S}v_u^{\equiv 2}\left(3^k,b\right)\\
&=2\cdot 3^{\left\lfloor k/2\right\rfloor}+\sum_{l=1}^{\left\lfloor\frac{k-1}{2}\right\rfloor}4\cdot 3^l\cdot 3^{k-2l-1}\\
&=2\cdot 3^{\left\lfloor k/2\right\rfloor}+4\cdot 3^{k-1}\sum_{l=1}^{\left\lfloor\frac{k-1}{2}\right\rfloor}3^{-l}\\
&=2\cdot 3^{\left\lfloor k/2\right\rfloor}+2\cdot 3^{k-1}-2\cdot 3^{k-\lfloor(k+1)/2\rfloor}\\
&=2\cdot 3^{k-1}.
\end{align*}
where we used the fact that $k=\lfloor(k+1)/2\rfloor+\left\lfloor k/2\right\rfloor$. Since there are exactly $2\cdot 3^{k-1}$ positive integers $n$ such that $0\leq n\leq 8\cdot 3^{k-1}-1$ with $n\equiv 2\pmod 4$ it follows that for any residue $0\leq b\leq 3^k-1$ not covered in the first two lines of \eqref{thm1case1u} we have $v_u\left(3^k,b\right)=v_u^{\not\equiv 2}\left(3^k,b\right)=2$. \eqref{thm1case1v} can be argued similarly.
\newline
\newline
2) The first two lines of \eqref{thm1case2u} follow from Propositions \ref{prop1}, \ref{prop2}, and \ref{prop3}. The third line follows from Proposition \ref{prop3}. It remains to show the fourth line. As in the proof of 1), we can argue that any residue $b$ that is not accounted for the first three lines we have $v_u\left(3^k,b\right)=v_u^{\not\equiv 2}\left(3^k,b\right)=0$. \eqref{thm1case2v} follows similarly.
\newline
\newline
3) Notice that for all $n\in\mathbb{N}$ we have $u_{8n+2}-u_2=Du_{4n}v_{4n+2}$ by \eqref{lem2.2}. By \eqref{lem3.2} and \eqref{lem3.4} we have $3^{\delta}\mid u_{4n}$ and $3^{\delta}\mid u_{4n+2}$. Thus $3^{2\delta}\mid u_{8n+2}-u_2$, so that $u_{8n+2}\equiv P\pmod{3^k}$. We can similarly derive that $u_{8n+6}\equiv -P\pmod{3^k}$. By these observations and Proposition \ref{prop3} we can see that $u_n\equiv P\pmod{3^k}$ if and only if $n\equiv 2\pmod 8$ and $u_n\equiv -P\pmod{3^k}$ if and only if $n\equiv 6\pmod 8$. The first line of \eqref{thm1case3u} follows. Moreover, the first line exactly accounts for all of the $u_n$ terms with $n\equiv 2\pmod 4$. Thus, the second and third lines also follow from Proposition \ref{prop3}. \eqref{thm1case3v} follows similarly.
\end{proof}

\section{The Case of $3\mid P$}

We now deal with the case $3\mid P$.

\begin{note}
For this section let $\delta:=\ord_3P$.
\end{note}

\begin{lemma}\label{lem10}
We have
\begin{align}
2\mid n&\Leftrightarrow 3\mid u_n\label{lem10.1}\\
2\mid n&\Rightarrow\ord_3 u_n=\ord_3 n+\ord_3P\label{lem10.2}\\
2\nmid n&\Leftrightarrow 3\mid v_n\label{lem10.3}\\
2\nmid n&\Leftrightarrow \ord_3v_n=\ord_3n+\ord_3P\label{lem10.4}.
\end{align}
\end{lemma}

\begin{proof}
From $u_0=0$, $u_1=1$, $u_2=P\equiv 0\pmod 3$, and $u_3\equiv 0^2+1\equiv 1\pmod 3$ we have \eqref{lem10.1}. From $v_0=2$, $v_1=P\equiv 0\pmod 3$, $v_2\equiv 0^2+2\equiv 2\pmod 3$, $v_3\equiv 2P\equiv 0\pmod 3$ we can see that \eqref{lem10.3} holds.
\newline
\newline
For \eqref{lem10.2} and \eqref{lem10.4} we prove by induction on $\ord_3(n)\in\mathbb{N}_0$. First suppose $2\mid n$, but $\ord_3n=0$. Then $n\equiv 2,4\pmod{6}$. Since $3\mid u_2$, we have by \eqref{lem1.4} that $9\mid u_{6}$. By \eqref{lem1.6} we have $\gcd(u_n,u_{6})=\gcd(u_4,u_{6})=\gcd(u_2,u_{6})$. Noting that $u_2=P$, we have $\ord_3P=\ord_3(u_4)=\ord_3(u_n)$, verifying $\eqref{lem3.2}$ for $\ord_3(n)=0$. Now for all $n\in\mathbb{N}$ such that $3\mid u_n$ we can see by \eqref{lem1.4} that $\ord_3(u_{3n})=\ord_3(u_n)+1$ and \eqref{lem10.2} now follows by induction on $\ord_3(n)$.
\newline
\newline
Now suppose that $2\nmid n$. By \eqref{lem1.1} we have $v_n=\frac{u_{2n}}{u_n}$. Since $2\nmid n$ we have $3\nmid u_n$ by \eqref{lem10.1}. Thus, $\ord_3\left(v_n\right)=\ord_3\left(u_{2n}\right)=\ord_3 n+\ord_3P$ by \eqref{lem10.2}.
\end{proof}
\begin{lemma}
For all $k\geq\delta$ we have
\begin{equation*}
h_u\left(3^k\right)=h_v\left(3^k\right)=2\cdot 3^{k-\delta}.
\end{equation*}
\end{lemma}
\begin{proof}
Let $k\geq\delta$. From \eqref{lem10.1} we have $2\mid h_u\left(3^k\right)$. For all $m\in\mathbb{Z}$ we have $3^k\mid u_{2m}$ if and only if $3^{k-\delta}\mid m$ by \eqref{lem10.2}. By \eqref{lem2.1} we have $u_{2\cdot 3^{k-\delta}+1}-u_1=u_{3^{k-\delta}+1}v_{3^{k-\delta}}$. By \eqref{lem10.4} $3^k\mid u_{2\cdot 3^{k-\delta}+1}-u_1$. It follows that $h_u\left(3^k\right)=2\cdot 3^{k-\delta}$.
\newline
\newline
From \eqref{lem10.3} we have $2\mid h_v\left(3^k\right)$. For all odd $m\in\mathbb{Z}$ we have $v_{2m}-v_0=v_{m}^2$ by \eqref{lem2.4}. From \eqref{lem10.4} we have $3^k\mid v_{m}^2$ if and only if $3^{\lceil k/2\rceil}\mid v_{m}$ if and only if $3^{\lceil k/2\rceil-\delta}\mid m$. By \eqref{lem2.1} we have $v_{2\cdot 3^{k-\delta}+1}-v_1=v_{3^{k-\delta}+1}v_{3^{k-\delta}}$. By \eqref{lem10.4} $3^k\mid u_{2\cdot 3^{k-\delta}+1}-u_1$. It follows that $h_u\left(3^k\right)=2\cdot 3^{k-\delta}$.
\end{proof}
\begin{note}
Let $x=2\delta-1$ and $J(k):=3^{2\lfloor(k-x)/4\rfloor+1}$.
\end{note}
\begin{lemma}\label{lem12}
Let $k\in\mathbb{N}$, $k\geq 2\delta-1$, and $n\in\mathbb{N}$. Then
\begin{equation}\label{lem12.1}
u_n\equiv u_{J(k)}\pmod{3^k}
\end{equation}
if and only if $n=J(k)+2\cdot 3^{\left\lfloor(k-x)/2\right\rfloor}j$ for some $j\in\mathbb{Z}$. Also,
\begin{equation}\label{lem12.2}
v_n\equiv v_0\pmod{3^k}.
\end{equation}
if and only if $n=2\cdot 3^{\left\lfloor(k-x)/2\right\rfloor}j$ for some $j\in\mathbb{Z}$.
\end{lemma}
\begin{proof}
First suppose $j$ is even. By \eqref{lem2.2} we have
\begin{equation*}
u_{J(k)+2\cdot 3^{\left\lfloor(k-x)/2\right\rfloor}j}-u_{J(k)}=u_{3^{\left\lfloor(k-x)/2\right\rfloor}j}v_{J(k)+3^{\left\lfloor(k-x)/2\right\rfloor}j}.
\end{equation*}
By \eqref{lem10.2} we have $\ord_3\left(u_{3^{\left\lfloor(k-x)/2\right\rfloor}j}\right)=\left\lfloor(k-x)/2\right\rfloor+\delta$. Also, by \eqref{lem10.4} we have
$\ord_3\left(v_{J(k)+3^{\left\lfloor(k-x)/2\right\rfloor}j}\right)\geq\left\lfloor(k-x)/2\right\rfloor+\delta$. Thus, $\ord_3\left(u_{J(k)+2\cdot 3^{(k-x)/2}j}-u_{J(k)}\right)\geq 2\left\lfloor(k-x)/2\right\rfloor+2\delta\geq k-x-1+2\delta=k$. Thus, \eqref{lem12.1} follows.
\newline
\newline
Now suppose $j$ is odd. By \eqref{lem2.1} we have
\begin{equation*}
u_{J(k)+2\cdot 3^{(k-x)/2}j}-u_{J(k)}=u_{J(k)+3^{\left\lfloor(k-x)/2\right\rfloor}j}v_{3^{\left\lfloor(k-x)/2\right\rfloor}j}.
\end{equation*}
By \eqref{lem10.2} we have $\ord_3\left(u_{J(k)+3^{\left\lfloor(k-x)/2\right\rfloor}}\right)=\left\lfloor(k-x)/2\right\rfloor+\delta$. Also, by \eqref{lem10.4} we have
$\ord_3\left(v_{3^{\left\lfloor(k-x)/2\right\rfloor}j}\right)\geq\left\lfloor(k-x)/2\right\rfloor+\delta$. Thus, $\ord_3\left(u_{J(k)+2\cdot 3^{(k-x)/2}j}-u_{J(k)}\right)=2\left\lfloor(k-x)/2\right\rfloor+2\delta\geq k-x-1+2\delta=k$. Thus, \eqref{lem12.1} follows.
\newline
\newline
For the converse, suppose $\eqref{lem12.1}$ holds for some $n$. Since $J(k)$ is odd, we have $3\nmid u_{J(k)}$ by \eqref{lem10.1}. So $3\nmid u_n$ and $n$ is odd again by \eqref{lem10.1}. Since $3^k\mid u_n-u_{J(k)}$, we have that $3^{\lfloor(k+1)/2\rfloor}$ divides at least one of the following by \eqref{lem2.1} and \eqref{lem2.2}: $u_{\frac{n+3^{2\lfloor(k-x/4\rfloor+1}}{2}}$, $u_{\frac{n-3^{2\lfloor(k-x/4\rfloor+1}}{2}}$,$v_{\frac{n+3^{2\lfloor(k-x/4\rfloor+1}}{2}}$, $v_{\frac{n-3^{2\lfloor(k-x/4\rfloor+1}}{2}}$. In all four cases it follows that either $3^{\lfloor(k+1)/2\rfloor-\delta}\mid n+3^{2\lfloor(k-x)/4\rfloor+1}$ or $3^{\lfloor(k+1)/2\rfloor-\delta}\mid n-3^{2\lfloor(k-x)/4\rfloor+1}$ by Lemma \ref{lem10}. Since $2\lfloor(k-x)/4\rfloor+1\geq\left\lfloor(k-x)/2\right\rfloor=\lfloor(k+1)/2\rfloor-\delta$ it follows that $3^{\left\lfloor(k-x)/2\right\rfloor}\mid n$.
\newline
\newline
By \eqref{lem2.4} we have 
\begin{equation*}
v_{2\cdot 3^{\left\lfloor(k-x)/2\right\rfloor}j}-v_0=Du_{3^{\left\lfloor(k-x)/2\right\rfloor}j}^2
\end{equation*}
if $j$ is even and by \eqref{lem2.3} we have
\begin{equation*}
v_{2\cdot 3^{\left\lfloor(k-x)/2\right\rfloor}j}-v_0=v_{3^{\left\lfloor(k-x)/2\right\rfloor}j}^2
\end{equation*}
if $j$ is odd. Using $\ord_3\left(u_{3^{\left\lfloor(k-x)/2\right\rfloor}j}\right)=\left\lfloor(k-x)/2\right\rfloor+\delta=\lfloor(k+1)/2\rfloor$ if $j$ is even and $\ord_3\left(v_{3^{\left\lfloor(k-x)/2\right\rfloor}j}\right)=\left\lfloor(k-x)/2\right\rfloor+\delta=\lfloor(k+1)/2\rfloor$ if $j$ is odd we obtain \eqref{lem12.2}.
\newline
\newline
For the converse, suppose \eqref{lem12.2} holds for some even $n$. Since $3\nmid v_0$, we have that $3\nmid v_n$. Thus, $n$ is even by \eqref{lem10.3}. Let $n=2m$. By \eqref{lem2.4} we have 
\begin{equation*}
v_n-v_0=Du_m^2
\end{equation*}
if $m$ even and by \eqref{lem2.3} we have
\begin{equation*}
v_n-v_0=v_m^2
\end{equation*}
is $m$ is odd. Thus, $3^{\lfloor(k+1)/2\rfloor}\mid u_m$ or $3^{\lfloor(k+1)/2\rfloor}\mid v_m$. It follows that $m=3^{\left\lfloor(k-x)/2\right\rfloor}j$ for some $j\in\mathbb{Z}$ by \eqref{lem3.1}. Thus, $n=2\cdot 3^{\left\lfloor(k-x)/2\right\rfloor}j$.
\end{proof}
\begin{prop}\label{prop4}
Let $k\in\mathbb{N}$, $k\geq 2\delta-1$. Then $v_u\left(3^k,u_{J(k)}\right)=v_v\left(3^k,v_0\right)=3^{\left\lfloor k/2\right\rfloor}$.
\end{prop}
\begin{proof}
This follows directly from Lemma \ref{lem12} noting that
\begin{equation*}
\frac{2\cdot 3^{k-\delta}}{2\cdot 3^{\left\lfloor(k-x)/2\right\rfloor}}=3^{\left\lfloor k/2\right\rfloor}.
\end{equation*}
\end{proof}
\begin{lemma}\label{lem13}
Let $l,n\in\mathbb{N}$, $l\geq\delta$. Then
\begin{equation}\label{lem13.1}
u_n\equiv u_{3^{l-\delta}}\pmod{3^{2l+1}}
\end{equation}
if and only if $n\equiv\pm 3^{l-\delta}\pmod{2\cdot 3^{l-\delta+1}}$. Also,
\begin{equation}\label{lem13.2}
v_n\equiv v_{2\cdot 3^{l-\delta}}\pmod{3^{2l+1}}.
\end{equation}
if and only if $n\equiv\pm 2\cdot 3^{l-\delta}\pmod{2\cdot 3^{l-\delta+1}}$.
\end{lemma}
\begin{proof}
If $i=1,5$ and $j$ is even, then we have
\begin{equation*}
\ord_3\left(u_{3^{l-\delta}i+2\cdot 3^{l-\delta+1}j}-u_{3^{l-\delta}}\right)=\ord_3\left(u_{\frac{3^{l-\delta}(i-1)}{2}+3^{l-\delta+1}j}v_{\frac{3^{l-\delta}(i+1)}{2}+3^{l-\delta+1}j}\right)\geq 2l+1
\end{equation*}
by \eqref{lem2.2}, \eqref{lem3.2}, and \eqref{lem3.4}. More specifically, if $i=1$, then $\ord_3\left(u_{\frac{3^{l-\delta}(i-1)}{2}+3^{l-\delta+1}j}\right)\geq l+1$ and 
\newline
$\ord_3\left(v_{\frac{3^{l-\delta}(i+1)}{2}+3^{l-\delta+1}j}\right)=l$, while if $i=5$, then $\ord_3\left(u_{\frac{3^{l-\delta}(i-1)}{2}+3^{l-\delta+1}j}\right)=l$ and
\newline
$\ord_3\left(v_{\frac{3^{l-\delta}(i+1)}{2}+3^{l-\delta+1}j}\right)\geq l+1$.
\newline
\newline
Similarly, if $i=1,5$ and $j$ is odd, then we have
\begin{equation*}
\ord_3\left(u_{2\cdot 3^{l-\delta}i+2\cdot 3^{l-\delta+1}j}-u_{2\cdot 3^{l-\delta}}\right)=\ord_3\left(v_{3^{l-\delta}(i-1)+3^{l-\delta+1}j}u_{3^{l-\delta}(i+1)+3^{l-\delta+1}j}\right)\geq 2l+1
\end{equation*}
by \eqref{lem2.1}, \eqref{lem3.2}, and \eqref{lem3.4}. Hence, we have \eqref{lem13.1}.
\newline
\newline
For the converse, suppose $\eqref{lem13.1}$ holds for some $n$. Since $3^{l-\delta}$ is odd, we have $3\nmid u_{3^{l-\delta}}$ by \eqref{lem10.1}. So $3\nmid u_n$ and $n$ is odd again by \eqref{lem10.1}. Since $3^{2l+1}\mid u_n-u_{3^{l-\delta}}$, we have that $3^{l+1}$ divides at least one of the following by \eqref{lem2.1} and \eqref{lem2.2}: $u_{\frac{n+3^{l-\delta}}{2}}$, $u_{\frac{n-3^{l-\delta}}{2}}$,$v_{\frac{n+3^{l-\delta}}{2}}$, $v_{\frac{n-3^{l-\delta}}{2}}$. In all four cases it follows that either $3^{l+1-\delta}\mid n+3^{l-\delta}$ or $3^{l+1-\delta}\mid n-3^{l-\delta}$ by Lemma \ref{lem10}. Thus, $3^{l-\delta}\mathrel\Vert n$ and since $n$ is odd, $n=\pm 3^{l-\delta}\pmod{2\cdot 3^{l-\delta+1}}$.
\newline
\newline
Similarly, if $i=\pm 1$ and $j$ even, we have
\begin{equation*}
\ord_3\left(v_{2\cdot 3^{l-\delta}i+2\cdot 3^{l-\delta+1}j}-v_{2\cdot 3^{l-\delta}}\right)=\ord_3\left(Du_{3^{l-\delta}(i-1)+3^{l-\delta+1}j}u_{3^{l-\delta}(i+1)+3^{l-\delta+1}j}\right)\geq 2l+1
\end{equation*}
and if $i=\pm 1$ and $j$ is odd, we have
\begin{equation*}
\ord_3\left(v_{2\cdot 3^{l-\delta}i+2\cdot 3^{l-\delta+1}j}-v_{2\cdot 3^{l-\delta}}\right)=\ord_3\left(v_{3^{l-\delta}(i-1)+3^{l-\delta+1}j}v_{3^{l-\delta}(i+1)+3^{l-\delta+1}j}\right)\geq 2l+1
\end{equation*}
by \eqref{lem2.3}, \eqref{lem2.4}, \eqref{lem3.2}, and \eqref{lem3.4} with both cases considered separately as before. Hence, we have \eqref{lem13.2}.
\newline
\newline
For the converse, suppose $\eqref{lem13.2}$ holds for some $n$. We have $3\nmid v_{2\cdot 3^{l-\delta}}$ by \eqref{lem10.1}. So $3\nmid v_n$ and $n$ is even again by \eqref{lem10.3}. Since $3^{2l+1}\mid v_n-v_{3^{l-\delta}}$, we have that $3^{l+1}$ divides at least one of the following by \eqref{lem2.3} and \eqref{lem2.4}: $u_{\frac{n+2\cdot 3^{l-\delta}}{2}}$, $u_{\frac{n-2\cdot 3^{l-\delta}}{2}}$,$v_{\frac{n+2\cdot 3^{l-\delta}}{2}}$, $v_{\frac{n-2\cdot 3^{l-\delta}}{2}}$. In all four cases it follows that either $3^{l+1-\delta}\mid n+2\cdot 3^{l-\delta}$ or $3^{l+1-\delta}\mid n-2\cdot 3^{l-\delta}$ by Lemma \ref{lem10}. Thus, $3^{l-\delta}\mathrel\Vert n$ and since $n$ is even, $n=\pm 2\cdot 3^{l-\delta}\pmod{2\cdot 3^{l-\delta+1}}$.
\end{proof}
\begin{prop}\label{prop5}
Let $k\geq 2\delta-1$.
\newline
1) Let $n\in\mathbb{N}$ be such that $u_n\equiv u_{3^{l-\delta}}\pmod{3^{2l+1}}$, where $\delta\leq l\leq\lfloor\frac{k-1}{2}\rfloor$. Then $v_u\left(3^k,u_n\right)=2\cdot 3^l$.
\newline
\newline
2) Let $n\in\mathbb{N}$ be such that $v_n\equiv v_{2\cdot 3^{l-\delta}}\pmod{3^{2l+1}}$, where $\delta\leq l\leq\lfloor\frac{k-1}{2}\rfloor$. Then $v_u\left(3^k,v_n\right)=2\cdot 3^l$.
\end{prop}
\begin{proof}
1) Let $n\in\mathbb{N}$ be such that $u_n\equiv u_{3^{l-\delta}}\pmod{3^{2l+1}}$. By Lemma \ref{lem13} $n$ has the form $a\cdot 3^{l-\delta}+2\cdot 3^{l-\delta+1}j$, where $a=1,5$. If $u_n\equiv u_m\pmod{3^k}$, then $u_n\equiv u_m\pmod{3^{2l+1}}$, so that $m$ must be of the form $b\cdot 3^{l-\delta}+2\cdot 3^{l-\delta+1}r$, where $b=1,5$, again by Lemma \ref{lem13}. For all such values of $m$ we have
\begin{equation}\label{prop5case1}
u_n-u_m=u_{\frac{(a-b)\cdot 3^{l-\delta}}{2}+3^{l-\delta+1}(j-r)}v_{\frac{(a+b)\cdot 3^{l-\delta}}{2}+3^{l-\delta+1}(j+r)}
\end{equation}
or
\begin{equation}\label{prop5case2}
u_n-u_m=v_{\frac{(a-b)\cdot 3^{l-\delta}}{2}+3^{l-\delta+1}(j-r)}u_{\frac{(a+b)\cdot 3^{l-\delta}}{2}+3^{l-\delta+1}(j+r)},
\end{equation}
depending on the parity of $\frac{(a-b)\cdot 3^{l-\delta}}{2}+3^{l-\delta+1}(j-r)$. First suppose that $a=b$. Then \eqref{prop5case1} and \eqref{prop5case2} simplify down to
\begin{equation}\label{prop5case1part2}
u_n-u_m=u_{3^{l-\delta+1}(j-r)}v_{a\cdot 3^{l-\delta}+3^{l-\delta+1}(j+r)}
\end{equation}
or
\begin{equation}\label{prop5case2part2}
u_n-u_m=v_{3^{l-\delta+1}(j-r)}u_{a\cdot 3^{l-\delta}+3^{l-\delta+1}(j+r)}.
\end{equation}
Suppose \eqref{prop5case1part2} holds. Then $j-r$ is even, and we can deduce that $3^k\mid u_n-u_m$ if and only if $3^{k-2l-1}\mid j-r$. Suppose \eqref{prop5case2part2}. Then $j-r$ is odd, and we can deduce that $3^k\mid u_n-u_m$ if and only if $3^{k-2l-1}\mid j-r$. Combining both of these possibilities gives $3^l$ possible values of $m$ with $a=b$. The case of $a\neq b$ is similar also leading to $3^l$ possible values of $m$ so Statement 1) follows. Statement 2) follows similarly, using \eqref{lem2.4}, \eqref{lem3.3}, \eqref{lem3.4}, and Lemma \ref{lem13}.
\end{proof}
\begin{lemma}\label{lem14}
Let $k\geq\delta$ and $0\leq n\leq 2\cdot 3^{k-\delta}-1$ with $n$ even ($n$ odd, respectively) and suppose $u_n\equiv b\pmod{3^k}$ ($v_n\equiv b\pmod{3^k}$, respectively), where $0\leq b\leq 3^k-1$. Then $u_{n+2\cdot 3^{k-\delta}j}$ ($v_{n+2\cdot 3^{k-\delta}j}$, respectively), $j=0,1,2$ are congruent to $b+3^k\lambda$, $\lambda=0,1,2$, modulo $3^{k+1}$ in some order.
\end{lemma}
\begin{proof}
By Lemma \ref{lem2} we have either
\begin{equation}\label{lem14case1}
u_{n+2\cdot 3^{k-\delta}j}-u_{n+2\cdot 3^{k-\delta}i}=u_{3^{k-\delta}(j-i)}v_{n+3^{k-\delta}(j+i)}
\end{equation}
or
\begin{equation}\label{lem14case2}
u_{n+2\cdot 3^{k-\delta}j}-u_{n+2\cdot 3^{k-\delta}i}=v_{3^{k-\delta}(j-i)}u_{n+3^{k-\delta}(j+i)}
\end{equation}
and either
\begin{equation}\label{lem14case3}
v_{n+2\cdot 3^{k-\delta}j}-v_{n+2\cdot 3^{k-\delta}i}=Du_{3^{k-\delta}(j-i)}u_{n+3^{k-\delta}(j+i)}
\end{equation}
or
\begin{equation}\label{lem14case4}
v_{n+3^{k-\delta}j}-v_{n+3^{k-\delta}i}=v_{3^{k-\delta}(j-i)}v_{n+3^{k-\delta}(j+i)}
\end{equation}
for all pairs of integers $0\leq i<j\leq 2$. First suppose that $n$ is even. Suppose that \eqref{lem14case1} holds. Then $3^{k-\delta}(j-i)$ is even so $n+3^{k-\delta}(j+i)$ is also even. Hence $3\nmid v_{n+3^{k-\delta}(j+i)}$ by \eqref{lem10.3}. Also, $3^{k-\delta}\mathrel\Vert 3^{k-\delta}(j-i)$, so that $3^{k}\mathrel\Vert u_{3^{k-\delta}(j-i)}$. Therefore, $3^k\mathrel\Vert u_{n+2\cdot 3^{k-\delta}j}-u_{n+2\cdot 3^{k-\delta}i}$. Suppose that \eqref{lem14case2} holds. Then $3^{k-\delta}(j-i)$ is odd so $n+3^{k-\delta}(j+i)$ is also odd. Hence $3\nmid u_{n+3^{k-\delta}(j+i)}$ by \eqref{lem10.1}. Also, $3^{k-\delta}\mathrel\Vert 3^{k-\delta}(j-i)$, so that $3^{k}\mathrel\Vert v_{3^{k-\delta}(j-i)}$. Therefore, $3^k\mathrel\Vert u_{n+2\cdot 3^{k-\delta}j}-u_{n+2\cdot 3^{k-\delta}i}$. The result on $n$ being even follows. The case of $n$ being odd is the same.
\end{proof}
\begin{prop}\label{prop6}
For all $k\geq\delta$, if $b\equiv 0\pmod{3^{\delta}}$, then $v_u\left(3^k,b\right)=v_v\left(3^k,b\right)=2$.
\end{prop}
\begin{proof}
First, using the fact that $3^\delta\mid P$ we can see that $u_n\equiv 0\pmod{3^{\delta}}$ if $n$ is even, $u_n\equiv 1\pmod{3^{\delta}}$ if $n$ is odd, $v_n\equiv 0\pmod{3^{\delta}}$ if $n$ is odd, and $v_n\equiv 2\pmod{3^{\delta}}$ if $n$ is even. Hence, the result holds for $k=\delta$. Going through an induction argument using Lemma \ref{lem14} gives the result for all $k\geq\delta$. 
\end{proof}
We now prove Theorem \ref{thm2}.
\newline
\newline
1) In this case, the result follows from Propositions \ref{prop4}, \ref{prop5}, \ref{prop6}, as long as we show that all odd indices are accounted for in the first two lines of \eqref{thm2case1u} and all even indices are accounted for in the first two lines of \eqref{thm2case1v}. Let $S$ be the set of residues that are accounted for by the first two lines of \eqref{thm2case1u}. By Propositions \ref{prop4} and \ref{prop5} we have
\begin{align*}
&\quad\sum_{b\in S}v_u\left(3^k,b\right)\\
&=3^{\left\lfloor k/2\right\rfloor}+\sum_{l=\delta}^{\left\lfloor\frac{k-1}{2}\right\rfloor}2\cdot 3^l\cdot 3^{k-2l-\delta}\\
&=3^{\left\lfloor k/2\right\rfloor}+2\cdot 3^{k-1}\sum_{l=\delta}^{\left\lfloor\frac{k-1}{2}\right\rfloor}3^{-l}\\
&=3^{k-\delta}.
\end{align*}
Since there are exactly $3^{k-\delta}$ even integers $n$ such that $0\leq n\leq 2\cdot 3^{k-\delta}-1$ \eqref{thm2case1u} follows. \eqref{thm2case1v} can be argued similarly.
\newline
\newline
2) Notice that for all $n\in\mathbb{N}$ we have $u_{4n+1}-u_1=u_{2n}v_{2n+1}$ by \eqref{lem2.2}. By \eqref{lem10.2} and \eqref{lem10.4} we have $3^{\delta}\mid u_{2n}$ and $3^{\delta}\mid v_{2n+1}$. Thus $3^{2\delta}\mid u_{4n+1}-u_1$, so that $u_{4n+1}\equiv 1\pmod{3^k}$. We can similarly derive that $u_{4n+3}\equiv u_1\pmod{3^k}$, using \eqref{lem2.1}, \eqref{lem10.2}, and \eqref{lem10.4}. By these observations and Proposition \ref{prop6} we have \eqref{thm2case2u}. \eqref{thm2case2v} follows similarly.
\section{Future Work}
There are still many questions left unanswered. For instance what can be said about Lucas sequences with recurrence relation $u_n=Pu_{n-1}+Qu_{n-2}$, where $Q\neq 1$? This problem appears to be more challenging because the analogous equations in Lemma \ref{lem2} in this more general setting do not offer a direct way to calculate the $p$-adic valuation of the difference between two terms in the sequence for a given prime $p$ (except for the case of $Q=-1$, which we leave as an exercise to the reader). Also, while there is a lot of progress on the residues modulo prime powers of second-order linear recurrence questions there is considerably less research on the analogous questions for higher order sequences. In fact, it was only in the last few years that real progress was made on the $p$-adic valuation of the terms themselves in higher order sequences \cite{antony},\cite{bilu},\cite{bravo},\cite{guadalupe},\cite{irmak},\cite{marques}. Again, the problem here is that there is no analogous result to Lemma \ref{lem2} for higher order sequences so finding some alternative approach will be necessary here too.

\section{Acknowledgements}

This paper was supported by MT Internal Grant Opportunities from Middle Tennessee State University.

\end{document}